\documentclass[onefignum,onetabnum]{siamart190516}

\usepackage{lipsum}
\usepackage{amsfonts}
\usepackage{graphicx}
\usepackage{epstopdf}
\usepackage{algorithmic}

\usepackage{graphicx}
\usepackage{epstopdf}
\usepackage{lineno,hyperref}
 \usepackage{geometry}
\usepackage{epsfig}
\usepackage{lineno,hyperref}
\usepackage{bookmark}
\usepackage{multirow}
 \usepackage{amsfonts}
\usepackage{multirow}
 \usepackage{algorithm}
 \usepackage[figuresright]{rotating}
 \usepackage{amscd}
\usepackage{mathrsfs}
\usepackage{booktabs,longtable}
\usepackage{indentfirst}
\usepackage{subfigure}
\usepackage{amstext}
\usepackage{amssymb}
\usepackage{fancyhdr}
\usepackage{amsmath}
\usepackage{cleveref}
\ifpdf
  \DeclareGraphicsExtensions{.eps,.pdf,.png,.jpg}
\else
  \DeclareGraphicsExtensions{.eps}
\fi

\newsiamremark{remark}{Remark}
\newsiamremark{hypothesis}{Hypothesis}
\crefname{hypothesis}{Hypothesis}{Hypotheses}
\newsiamthm{claim}{Claim}

\headers{Greedy Block Gauss-Seidel Methods}{Hanyu Li and Yanjun Zhang}

\title{Greedy Block Gauss-Seidel Methods for Solving Large Linear Least Squares Problem\thanks{Submitted to the editors DATE.
\funding{This work was funded by the National Natural Science Foundation of China (No. 11671060) and the Natural Science Foundation Project of CQ CSTC (No. cstc2019jcyj-msxmX0267)} }}

\author{Hanyu Li\thanks{College of Mathematics and Statistics, Chongqing University, Chongqing 401331, P.R. China
  (\email{lihy.hy@gmail.com or hyli@cqu.edu.cn; \ yjzhang@cqu.edu.cn}).}
\and Yanjun Zhang\footnotemark[2]}

\usepackage{amsopn}

\ifpdf
\hypersetup{
  pdftitle={Greedy Block Gauss-Seidel Methods for Solving Large Linear Least Squares Problem},
  pdfauthor={Hanyu Li and Yanjun Zhang}
}
\fi

\begin{document}

\maketitle

\begin{abstract}
 With a greedy strategy to construct control index set of coordinates firstly and then choosing the corresponding column submatrix in each iteration, we present a greedy block Gauss-Seidel (GBGS) method for solving large linear least squares problem. Theoretical analysis demonstrates that the convergence factor of the GBGS method can be much smaller than that of the greedy randomized coordinate descent (GRCD) method proposed recently in the literature. 
On the basis of the GBGS method, we further present a pseudoinverse-free greedy block Gauss-Seidel method, which doesn't need to calculate the Moore-Penrose pseudoinverse of the column submatrix in each iteration any more
and hence can be  achieved greater acceleration. Moreover, this method can also be used for distributed implementations.
Numerical experiments 
show that, for the same accuracy, our methods can far outperform the GRCD method in terms of the iteration number and computing time.
\end{abstract}

\begin{keywords}
greedy block Gauss-Seidel method, pseudoinverse-free, greedy randomized coordinate descent method, large linear least squares problem
\end{keywords}

\begin{AMS}
  65F10, 65F20
\end{AMS}

\section{Introduction}\label{intro}
As we know, the linear least squares problem is a classical problem in numerical linear algebra and has numerous important applications in many fields. 
There are many direct methods for solving this problem such as 
the method by 
the QR factorization with pivoting and the method by 
the SVD \cite{Osborne1996,Higham2002}. However, these direct methods usually require high storage and are expensive when the involved matrix is large-scale. Hence, some iterative methods are provided for solving large linear least squares problem.  
The Gauss-Seidel method \cite{Saad2003} is a very famous one, which,  at each iteration,  
first selects a coordinate $j_k\in \{1, \ldots, n\}$ and then minimizes the objective $L(x)=\frac{1}{2} \|{b}-{A} {x}\|^2_{2}$ with respect to this coordinate. This leads to the following iterative process:
\begin{equation*}
\label{gs}
x_{k+1}=x_{k}+\frac{ A^{T}_{(j_k)}(b-Ax_{k}) }{ \| A_{\left(j_{k}\right)} \|_{2}^{2}}e_{j_k},
\end{equation*}
where $A_{\left(j_{k}\right)}$ denotes the $j_k$th column of $A\in R^{m\times n}$ and $e_{j_k}$ is the $j_k$th coordinate basis column vector, i.e., the $j_k$th column of the identity matrix $I$.

In 2010,  by using a probability distribution to select the column of $A$ randomly in each iteration, 
Leventhal and Lewis \cite{Leventhal2010} considered the randomized Gauss-Seidel (RGS) method, which is also called the randomized coordinate descent method, and showed that it converges linearly in expectation to the solution. Subsequently, many works on RGS method were reported 
due to its better performance; see, for example \cite{Ma2015,Hefny2017,Edalatpour2017,Tu2017,Chen2017,Tian2017,Xu1,Razaviyayn2019,Dukui2019} and references therein.

In 2018, Wu \cite{WEIMAGGIE2018CONVERGENCEOT} developed a randomized block Gauss-Seidel (RBGS) method, i.e., a block version of the RGS method, 
which 
minimizes the objective through multiple coordinates at the same time in each iteration. More specifically, for a subset $\tau \subset \{1, \ldots, n\}$ with $|\tau|=T$, i.e., the number of the elements of the set $\tau$ is $T$, denote the submatrix of $A$ whose columns are indexed by $\tau$ by $A_{\tau}$ and generate a submatrix $I_{\tau}$ of the identity matrix $I$ similarly.
The iterative process can be written as 
\begin{equation}
\label{rbgs}
x_{k+1}=x_{k}+I_{\tau}A^{\dagger}_{\tau}(b-Ax_{k}),
\end{equation}
where $A^{\dagger}_{\tau}$ denotes the Moore-Penrose pseudoinverse of $A_{\tau}$. The discussions in \cite{WEIMAGGIE2018CONVERGENCEOT} show that the RBGS method has better convergence compared with the RGS method. Later, Du and Sun \cite{du2019doubly} generalized the RGS method to a doubly stochastic block Gauss-Seidel method.

Recently, by introducing a more efficient probability criterion for selecting the working column from the matrix $A$, Bai and Wu \cite{Bai2019} proposed a greedy randomized coordinate descent (GRCD) method, 
which is faster than the RGS method in terms of the number of iterations and computing time. 
The idea of greed applied in \cite{Bai2019}  has wide applications and has been used in many works, see for example \cite{Griebel2012,Nguyen2017,Bai2018,Bai2018r,Nutini2018,Zhang2019,Du2019,Liu2019,Haddock2019,Rebrova2019,Niu2020} and references therein.

Inspired by the ideas of the RBGS and GRCD methods, 
we develop a 
greedy block Gauss-Seidel (GBGS) method in the present paper, and show that 
the convergence  factor of the GBGS method can be much smaller than that of the GRCD method. 
Experimental results also demonstrate that the GBGS method can significantly accelerate the GRCD method. A drawback of the GBGS method is that, in the block update rule like \cref{rbgs}, we have to 
compute the Moore-Penrose pseudoinverses in each iteration, which is quite expensive and also makes the method 
be not adequate for distributed implementations. To tackle this problem, inspired by the idea in \cite{Necoara2019,Du20202}, 
we present a pseudoinverse-free greedy block Gauss-Seidel (PGBGS) method, which 
improves the GBGS method in computation cost. 

The rest of this paper is organized as follows. In \cref{sec2}, additional notation and preliminaries are provided. We present our GBGS and PGBGS methods and their corresponding convergence properties in \cref{sec3} and \cref{sec4}, respectively. Experimental results are given in \cref{sec5}.
%
%
%
\section{Notation and preliminaries}\label{sec2}

For a vector $z\in R^{n}$, $z^{(j)}$ represents its $j$th entry. For a matrix $G=(g_{ij})\in R^{m\times n}$, 
$\sigma_{\min}(G)$, $\sigma_{\max}(G)$, $\|G\|_2$, $\|G\|_F$ and $\mathrm{range}(G)$ denote its 
minimum positive singular value, maximum singular value, spectral norm, Frobenius norm, and column space, respectively. 
Moreover, if the matrix $G\in R^{n\times n}$ is positive definite, then we define the energy norm of any vector $x\in R^{n}$ as $\| x\|_G:=\sqrt{x^TGx}$. 

In the following, 
we use $x_{\star}=A^{\dag}b$ 
to denote the unique least squares solution to the linear least squares problem:
\begin{equation}
\label{ls}
\min \limits _{\mathbf{x} \in \mathbb{R}^{n}}\|{b}-{A} {x}\|^2_{2},
\end{equation}
where $A\in R^{m\times n}$ is of full column rank and $b\in R^{m}$.
It is known that the solution $x_\star$ 
is the solution to the following normal equation \cite{Osborne1961} for \cref{ls}:
\begin{equation}
\label{1}
A^TAx=A^Tb.
\end{equation}
On the basis of  \cref{1}, Bai and Wu \cite{Bai2019} proposed the GRCD method listed in \cref{alg1}, where $r_k = b-Ax_k$ denotes the residual vector.
\begin{algorithm}
\caption{The GRCD method}
\label{alg1}
\begin{algorithmic}[!htb]
\STATE{\mbox{Input:} ~$A\in R^{m\times n}$, $b\in R^{m}$, $\ell$ , initial estimate $x_0$}
\STATE{\mbox{Output:} ~$x_\ell$}
\FOR{$k=0, 1, 2, \ldots, \ell-1$}
\STATE{Compute
\begin{align}
  \delta_{k}=\frac{1}{2}\left(\frac{1}{\left\|A^Tr_k\right\|_{2}^{2}} \max \limits _{1 \leq j  \leq n}\left\{\frac{\left|A^T_{(j)}r_k\right|^{2}}{\left\|A_{\left(j\right)}\right\|_{2}^{2}}\right\}+\frac{1 }{\|A\|_{F}^{2}}\right).\notag
\end{align}}
\STATE{Determine the index set of positive integers
\begin{align}
 \mathcal{V}_{k}=\left\{j \Bigg|  \left|A^T_{(j)}r_k\right|^{2} \geq \delta_{k} \left\|A^Tr_k\right\|_{2}^{2} \left\|A_{\left(j\right)}\right\|_{2}^{2}\right\}.\notag
\end{align}}
\STATE{Let $s_k=A^Tr_k$ and define $\tilde{s}_k$ as follows \[
\tilde{s}_{k}^{(j)}=\left\{\begin{array}{ll}{s_{k}^{(j)},} & {\text { if } j \in \mathcal{V}_{k}}, \\ {0,} & {\text { otherwise. }}\end{array}\right.
\]}
\STATE{Select $j_k \in \mathcal{V}_{k}$ with probability $Pr(\mathrm{column} = j_k)=\frac{|\tilde{s}_{k}^{(j_k)}|^2}{\|\tilde{s}_{k}\|^2_2}$.}
\STATE{Set
\[x_{k+1}=x_{k}+\frac{ A^{T}_{(j_k)}r_{k} }{ \| A_{\left(j_{k}\right)} \|_{2}^{2}}e_{j_k}.\]}
\ENDFOR
\end{algorithmic}
\end{algorithm}

In addition, to analyze the convergence of our new methods, we need  the following simple result, which can be found in \cite{Bai2018}.
\begin{lemma}
\label{lem1}
For any vector $x \in  \mathrm{range} (A) $, it holds that
\begin{align*}
\|A^{T}x\|^2_2\geq\sigma^{2}_{\min}\left( A\right)\|x\|^2_2.  \label{1210}
\end{align*}
\end{lemma}

\section{Greedy block Gauss-Seidel method}\label{sec3}


There are two main steps in the GBGS method. The first step is to devise a greedy rule to decide the index set $\mathcal{J}_{k}$ whose specific definition is given in \cref{alg2}, and the second step is to update $x_k$ using a update formula like \cref{rbgs}, with which we can
minimize the objective function through all coordinates from the index set $\mathcal{J}_{k}$ at the same time. Note that the GRCD method only updates one coordinate in each iteration. 

Based on the above introduction, we propose 
\cref{alg2}.
\begin{algorithm}[!htb]
\caption{The GBGS method}
\label{alg2}
\begin{algorithmic}
\STATE{\mbox{Input:} ~$A\in R^{m\times n}$, $b\in R^{m}$, $\theta \in[0,1]$, $\ell$, 
 initial estimate $x_0\in R^{n}$}
\STATE{\mbox{Output:} ~$x_\ell$}
\FOR{$k=0, 1, 2, \ldots, \ell-1$}
\STATE{Compute
\begin{align}
\varepsilon_{k}=\frac{\theta}{\left\|A^Tr_k\right\|_{2}^{2}} \max \limits _{1 \leq j  \leq n}\left\{\frac{\left|A^T_{(j)}r_k\right|^{2}}{\left\|A_{\left(j\right)}\right\|_{2}^{2}}\right\}+\frac{1-\theta }{\|A\|_{F}^{2}} .\notag
\end{align}}
\STATE{Determine the index set of positive integers
\begin{align}
 \mathcal{J}_{k}=\left\{j_k \Bigg|  \left|A^T_{(j_k)}r_k\right|^{2} \geq \varepsilon_{k}\left\|A^Tr_k\right\|_{2}^{2}  \left\|A_{\left(j_k\right)}\right\|_{2}^{2}\right\}.\notag
\end{align}}
\STATE{Set
\begin{align}
\label{up2}
x_{k+1}=x_{k}+I_{\mathcal{J}_{k}}A^{\dag}_{\mathcal{J}_{k}}r_{k}.
\end{align}}
\ENDFOR
\end{algorithmic}
\end{algorithm}

\begin{remark}
\label{rmk1}
As done in \cite{Bai2018r}, if we replace $ \delta_{k}$ in \cref{alg1} by $\varepsilon_{k}$ in \cref{alg2}, we can get a relaxed version of the GRCD method.
In addition, it is easy to see that if $\theta=\frac{1}{2}$, then $\mathcal{J}_{k}=\mathcal{V}_{k}$, i.e., the index sets of the GBGS and GRCD methods are the same. 
\end{remark}

\begin{remark}
\label{rmk2}
Note that if
\[\frac{\left|A^T_{(j_k)}r_k\right|^{2}}{\left\|A_{\left(j_k\right)}\right\|_{2}^{2}}=  \max \limits _{1 \leq j  \leq n}\left\{\frac{\left|A^T_{(j)}r_k\right|^{2}}{\left\|A_{\left(j\right)}\right\|_{2}^{2}}\right\},\]
then $j_k\in\mathcal{J}_{k}.$ So the index set $\mathcal{J}_{k}$ in \cref{alg2} is nonempty for all iteration index $k$.
\end{remark}

In the following, we give the convergence theorem of the GBGS method.

\begin{theorem}
\label{theorem1}
The iteration sequence $\{x_k\}_{k=0}^\infty$ generated by \cref{alg2}, starting from an initial guess $x_0\in R^{n}$, converges linearly to the unique least squares solution $x_{\star}=A^{\dag}b$ and
\begin{equation}
\label{4}
  \| x_1-x_\star\|^2_{A^TA} \leq\left(1- \frac{\|A_{\mathcal{J}_{0}}\|^2_F}{\sigma^2_{\max}(A_{\mathcal{J}_{0}})}  \frac{\sigma^{2}_{\min}(A)}{\left\|A\right\|_{F}^{2}} \right) \| x_0-x_\star\|^2_{A^{T}A},
\end{equation}
and
 \begin{align}
\label{5}
\| x_{k+1}-x_\star\|^2_{A^TA} \leq\left(1- \frac{\|A_{\mathcal{J}_{k}}\|^2_F}{\sigma^2_{\max}(A_{\mathcal{J}_{k}})} \left(\theta\frac{\|A\|^2_F}{\|A\|^2_F-\|A_{\mathcal{J}_{k-1}}\|^2_F}+(1-\theta)\right) \frac{\sigma^{2}_{\min}(A)}{\|A\|^2_F}\right)\| x_k-x_\star\|^2_{A^TA},\\
k=1, 2, \ldots.\notag
\end{align}
Moreover, let $\alpha=\max \{\sigma^2_{\max}(A_{\mathcal{J}_{k}})\}$, $\beta=\min\{\|A_{\mathcal{J}_{k}}\|^2_F\}$, and $\gamma=\max\{\|A\|^2_F-\|A_{\mathcal{J}_{k-1}}\|^2_F\}$. 
Then
\begin{align}
\label{6}
\| x_{k}-x_\star\|^2_{A^TA} &\leq\left(1- \frac{\beta}{\alpha}(\theta\frac{\|A\|^2_F}{\gamma}+(1-\theta)) \frac{\sigma^{2}_{\min}(A)}{\|A\|^2_F} \right)^{k-1}\\
&\times\left(1- \frac{\|A_{\mathcal{J}_{0}}\|^2_F}{\sigma^2_{\max}(A_{\mathcal{J}_{0}})}\frac{\sigma^{2}_{\min}(A)}{\left\|A\right\|_{F}^{2}} \right)\| x_0-x_\star\|^2_{A^TA},
\quad k=1, 2, \ldots.\notag
\end{align}
\end{theorem}
\begin{proof}
 From the update formula \cref{up2} in \cref{alg2}, we have
 \begin{align}
A(x_{k+1}-x_{\star})&=~A(x_{k}-x_{\star})+A_{\mathcal{J}_{k}}A^{\dag}_{\mathcal{J}_{k}}r_{k},  \notag
  \end{align}
  which together with the fact $A_{\mathcal{J}_{k}}A^{\dag}_{\mathcal{J}_{k}}=(A_{\mathcal{J}_{k}}A^{\dag}_{\mathcal{J}_{k}})^{T}$ and $A^TAx_{\star}=A^Tb$ gives
\begin{align}
A(x_{k+1}-x_{\star})&= ~ A(x_{k}-x_{\star})+(A^{\dag}_{\mathcal{J}_{k}})^{T}A^{T}_{\mathcal{J}_{k}}(b-Ax_{k})  \notag\\
&= ~\left(I-A_{\mathcal{J}_{k}}A^{\dag}_{\mathcal{J}_{k}} \right)A\left(x_{k}-x_\star\right). \notag
\end{align}
Since $A_{\mathcal{J}_{k}}A^{\dag}_{\mathcal{J}_{k}}$ is an orthogonal projector, taking the square of the Euclidean norm on both sides of the above equation and applying Pythagorean theorem, we get
\begin{align}
\|A(x_{k+1}-x_{\star})\|^2_2&= ~\|\left(I-A_{\mathcal{J}_{k}}A^{\dag}_{\mathcal{J}_{k}} \right)A\left(x_{k}-x_\star\right)\|^2_2\notag\\
&= ~\|A\left(x_{k}-x_\star\right)\|^2_2-\|A_{\mathcal{J}_{k}}A^{\dag}_{\mathcal{J}_{k}}A\left(x_{k}-x_\star\right)\|^2_2\notag\\
&= ~\|A\left(x_{k}-x_\star\right)\|^2_2-\|(A^{\dag}_{\mathcal{J}_{k}})^{T}A^{T}_{\mathcal{J}_{k}}A\left(x_{k}-x_\star\right)\|^2_2,\notag
\end{align}
or equivalently,
\begin{align}
\left\|x_{k+1}-x_{\star}\right\|^2_{A^{T}A}=\left\|x_{k}-x_{\star}\right\|^2_{A^{T}A}-\|(A^{\dag}_{\mathcal{J}_{k}})^{T}A^{T}_{\mathcal{J}_{k}}A\left(x_{k}-x_\star\right)\|^2_2,\notag
\end{align}
which together with \cref{lem1} and the fact $\sigma^2_{\min}(A^{\dag}_{\mathcal{J}_{k}})=\sigma^{-2}_{\max}(A_{\mathcal{J}_{k}})$ yields
\begin{align}
\left\|x_{k+1}-x_{\star}\right\|^2_{A^{T}A}&\leq~\left\|x_{k}-x_{\star}\right\|^2_{A^{T}A}-\sigma^2_{\min}(A^{\dag}_{\mathcal{J}_{k}})\|A^{T}_{\mathcal{J}_{k}}A\left(x_{k}-x_\star\right)\|^2_2\notag\\
&=~\left\|x_{k}-x_{\star}\right\|^2_{A^{T}A}-\frac{1}{\sigma^2_{\max}(A_{\mathcal{J}_{k}})}\|A^{T}_{\mathcal{J}_{k}}A\left(x_{k}-x_\star\right)\|^2_2.\notag
\end{align}
On the other hand, from \cref{alg2}, we have
\begin{align}
\|A^{T}_{\mathcal{J}_{k}}A\left(x_{k}-x_\star\right)\|^2_2= \sum\limits_{j_k\in  \mathcal{J}_{k}}  \left|A^T_{(j_k)}r_k\right|^2
\geq\sum\limits_{j_k\in  \mathcal{J}_{k}} \varepsilon_{k} \left\|A^Tr_k\right\|_{2}^{2} \left\|A_{\left(j_k\right)}\right\|_{2}^{2}
= \varepsilon_{k}\left\|A^Tr_k\right\|_{2}^{2}  \|A_{\mathcal{J}_{k}}\|^2_F. \notag
\end{align}
Then,
\begin{align}
\left\|x_{k+1}-x_{\star}\right\|^2_{A^{T}A}\leq\left\|x_{k}-x_{\star}\right\|^2_{A^{T}A}-\frac{\|A_{\mathcal{J}_{k}}\|^2_F}{\sigma^2_{\max}(A_{\mathcal{J}_{k}})}\left\|A^Tr_k\right\|_{2}^{2} \varepsilon_{k}.\label{1217}
\end{align}

For $k=0$, we have
\begin{align}
\left\|A^Tr_0\right\|_{2}^{2} \varepsilon_{0}&=~\theta \max \limits _{1 \leq j  \leq n}\left\{\frac{\left|A^T_{(j)}r_0\right|^{2}}{\left\|A_{\left(j\right)}\right\|_{2}^{2}}\right\}+(1-\theta )\frac{ \left\|A^Tr_0\right\|_{2}^{2}}{\|A\|_{F}^{2}}\notag
\\
&\geq~\theta \sum\limits_{j=1}^{n} \frac{\left\|A_{\left(j\right)}\right\|_{2}^{2}}{\|A\|_{F}^{2}} \frac{\left|A^T_{(j)}r_0\right|^{2}}{\left\|A_{\left(j\right)}\right\|_{2}^{2}}+(1-\theta )\frac{ \left\|A^Tr_0\right\|_{2}^{2}}{\|A\|_{F}^{2}}\notag \\
&=~\theta\frac{\|A^Tr_0\|^{2}_2}{\left\|A\right\|_{F}^{2}}+(1-\theta )\frac{ \left\|A^Tr_0\right\|_{2}^{2}}{\|A\|_{F}^{2}}=\frac{ \left\|A^Tr_0\right\|_{2}^{2}}{\|A\|_{F}^{2}},\notag
\end{align}
which together with \cref{lem1} leads to
\begin{align}
\left\|A^Tr_0\right\|_{2}^{2} \varepsilon_{0}\geq  \frac{1}{\left\|A\right\|_{F}^{2}} \sigma^{2}_{\min}(A)\|A(x_0-x_\star)\|^{2}_2=  \frac{\sigma^{2}_{\min}(A)}{\left\|A\right\|_{F}^{2}} \|x_0-x_\star\|^{2}_{A^{T}A}.\label{1212111}
\end{align}
Thus, substituting \cref{1212111} into \cref{1217}, we obtain
\begin{align}
\left\|x_{1}-x_{\star}\right\|^2_{A^{T}A}&\leq~\left\|x_{0}-x_{\star}\right\|^2_{A^{T}A}-\frac{\|A_{\mathcal{J}_{0}}\|^2_F}{\sigma^2_{\max}(A_{\mathcal{J}_{0}})}  \frac{\sigma^{2}_{\min}(A)}{\left\|A\right\|_{F}^{2}} \|x_0-x_\star\|^{2}_{A^{T}A} \notag\\
&=~\left(1- \frac{\|A_{\mathcal{J}_{0}}\|^2_F}{\sigma^2_{\max}(A_{\mathcal{J}_{0}})}  \frac{\sigma^{2}_{\min}(A)}{\left\|A\right\|_{F}^{2}} \right) \| x_0-x_\star\|^2_{A^{T}A}, \notag
\end{align}
which is just the estimate \cref{4}.

For $k\geq1$, to find the lower bound of $\left\|A^Tr_k\right\|_{2}^{2} \varepsilon_{k}$, first note that
\begin{align}
 A^T_{ \mathcal{J}_{k-1} }r_k&=  ~A^T_{\mathcal{J}_{k-1}}\left(b-A(x_{k-1}+I_{\mathcal{J}_{k-1}}A^{\dag}_{\mathcal{J}_{k-1}}r_{k-1})  \right)   \notag\\
&= ~A^T_{ \mathcal{J}_{k-1}}r_{k-1}-A^T_{ \mathcal{J}_{k-1}} A_{\mathcal{J}_{k-1}}A^{\dag}_{\mathcal{J}_{k-1}}r_{k-1}  \notag\\
&= ~0, \notag
\end{align}
where we have used the update formula \cref{up2} and the property of the Moore-Penrose pseudoinverse.
Then
\begin{align}
\|A^{T}r_k\|^2_2&=~\sum\limits_{j=1}^{n}|A^{T}_{(j)}r_k|^2=\sum\limits_{  j=1 \atop  j\notin  \mathcal{J}_{k-1}}^n|A^{T}_{(j)}r_k|^2 =~\sum\limits_{  j=1 \atop  j\notin  \mathcal{J}_{k-1}}^n\frac{|A^{T}_{(j)}r_k|^2}{\|A_{j}\|^2_2} \|A_{j}\|^2_2\notag\\
&\leq~\max\limits_{1\leq j\leq n  }\frac{|A^{T}_{(j)}r_k|^2}{\|A_{j}\|^2_2}  \sum\limits_{  j=1 \atop  j\notin  \mathcal{J}_{k-1}}^n\|A_{j}\|^2_2=~\max\limits_{1\leq j\leq n  }\frac{|A^{T}_{(j)}r_k|^2}{\|A_{j}\|^2_2} (\|A\|^2_F-\|A_{\mathcal{J}_{k-1}}\|^2_F),\notag
\end{align}
which can first imply a lower bound of $\max\limits_{1\leq j\leq n  }\frac{|A^{T}_{(j)}r_k|^2}{\|A_{j}\|^2_2}$ and then of $\left\|A^Tr_k\right\|_{2}^{2} \varepsilon_{k}$ as follows
\begin{align}
\left\|A^Tr_k\right\|_{2}^{2} \varepsilon_{k}
\geq~ \theta\frac{\|A^{T}r_k\|^2_2}{\|A\|^2_F-\|A_{\mathcal{J}_{k-1}}\|^2_F}+(1-\theta)\frac{ \left\|A^Tr_k\right\|_{2}^{2}}{\|A\|_{F}^{2}}.\notag
\end{align}
Further, considering 
\cref{lem1}, we have
\begin{align}
\left\|A^Tr_k\right\|_{2}^{2}\varepsilon_{k}&\geq~  \left(\theta\frac{1}{\|A\|^2_F-\|A_{\mathcal{J}_{k-1}}\|^2_F}+(1-\theta)\frac{ 1}{\|A\|_{F}^{2}}\right) \sigma^{2}_{\min}(A) \|A(x_k-x_\star)\|^2_2\notag\\
&=~ \left(\theta\frac{\|A\|^2_F}{\|A\|^2_F-\|A_{\mathcal{J}_{k-1}}\|^2_F}+(1-\theta)\right) \frac{\sigma^{2}_{\min}(A)}{\|A\|^2_F} \|x_k-x_\star\|^{2}_{A^{T}A}.\label{12121111}
\end{align}
Thus, substituting \cref{12121111} into \cref{1217} gives
the estimate \cref{5}. By induction on the iteration index $k$, we can obtain the estimate \cref{6}.
\end{proof}

\begin{remark}
\label{rmk3111}
If $\theta=\frac{1}{2}$, i.e., the index sets of the GBGS and GRCD methods are the same, the first term in the right side of \cref{5} reduces to
$$
\eta=1- \frac{\|A_{\mathcal{J}_{k}}\|^2_F}{\sigma^2_{\max}(A_{\mathcal{J}_{k}})} \frac{1}{2}\left( \frac{\|A\|^2_F}{\|A\|^2_F-\|A_{\mathcal{J}_{k-1}}\|^2_F}+1\right) \frac{\sigma^{2}_{\min}(A)}{\|A\|^2_F}.
$$
Since
\[ \|A\|^2_{F}-\|A_{\mathcal{J}_{k-1}}\|^2_{F}\leq\|A\|^2_F-\min\limits_{1\leq j\leq n}\|A_{(j)}\|^{2}_2\ \textrm{ and }\ \frac{\|A_{\mathcal{J}_{k}}\|^2_{F}}{\sigma^2_{\max}(A_{\mathcal{J}_{k}})}\geq1,\]
we have
\begin{eqnarray*}
\eta \leqslant 1- \frac{1}{2} \left(\frac{\|A\|_{F}^{2}}{\|A\|_{F}^{2}-\min \limits_{1 \leq j \leq n}\left\|A_{(j)}\right\|_{2}^{2}}+1 \right)\frac{\sigma^2_{\min}(A)}{\|A\|_{F}^{2}}.
\end{eqnarray*}
Note that the error estimate in expectation of the GRCD method given in \cite{Bai2019} is
\[
\mathbb{E}_{k}\left\|x_{k+1}-x_{\star}\right\|_{A^{T} A}^{2} \leq\left (1- \frac{1}{2} \left(\frac{\|A\|_{F}^{2}}{\|A\|_{F}^{2}-\min \limits_{1 \leq j \leq n}\left\|A_{(j)}\right\|_{2}^{2}}+1 \right)\frac{\sigma^2_{\min}(A)}{\|A\|_{F}^{2}}\right  )\left\|x_{k}-x_{\star}\right\|_{A^{T} A}^{2} , \]
where $ k=1,2, \ldots.$ So the convergence  factor of the GBGS method is 
smaller than that of the GRCD method in this case, 
and the former can be much smaller than the latter because
$\|A\|^2_{F}-\|A_{\mathcal{J}_{k-1}}\|^2_{F}$ can be much smaller than $\|A\|^2_F-\min\limits_{1\leq j\leq n}\|A_{(j)}\|^{2}_2$ and $\|A_{\mathcal{J}_{k}}\|^2_{F}$ can be much larger than $\sigma^2_{\max}(A_{\mathcal{J}_{k}})$. 

\end{remark}

\begin{remark}
\label{rmk4}
Very recently, Niu and Zheng \cite{Niu2020} provided a greedy block Kaczmarz algorithm for solving large-scale linear systems, which is a block version of the famous greedy randomized Kaczmarz method given in \cite{Bai2018}. The technique in \cite{Niu2020} can be also applied to develop a block version of the greedy Gauss-Seidel method.

\end{remark}

\section{Pseudoinverse-free greedy block Gauss-Seidel method}\label{sec4}
As mentioned in \cref{intro},  there is a drawback in the update rule \cref{up2} in \cref{alg2}, that is, 
we need to compute the Moore-Penrose pseudoinverse of the column submatrix $A_{\mathcal{J}_{k}}$ in each iteration. 
Inspired by the pseudoinverse-free block Kaczmarz methods \cite{Necoara2019, Du20202}, 
we design a PGBGS method. 
More specifically, assume 
 $|\mathcal{J}_{k}|=t$ and let $\mathcal{J}_{k}=\{j_{k1}, j_{k2}, \ldots, j_{kt}\}$. 
 Rather than using the update rule \cref{up2}, we execute all coordinates in $\mathcal{J}_{k}$ simultaneously using the following update rule
\begin{align}
x_{k+1}&=~x_{k}+\omega\frac{ A^{T}_{(j_{k1})}r_{k} }{ \| A_{\left(j_{k1}\right)} \|_{2}^{2}}e_{j_{k1}}+\omega\frac{ A^{T}_{(j_{k2})}r_{k} }{ \| A_{\left(j_{k2}\right)} \|_{2}^{2}}e_{j_{k2}}+\cdots+\omega\frac{ A^{T}_{(j_{kt})}r_{k} }{ \| A_{\left(j_{kt}\right)} \|_{2}^{2}}e_{j_{kt}} \notag\\
&=~x_{k}+\omega\frac{e_{j_{k1}}}{\| A_{\left(j_{k1}\right)} \|_{2}}\frac{ A^{T}_{(j_{k1})}}{ \| A_{\left(j_{k1}\right)} \|_{2}}r_{k}+\omega\frac{e_{j_{k2}}}{\| A_{\left(j_{k2}\right)} \|_{2}}\frac{ A^{T}_{(j_{k2})}}{ \| A_{\left(j_{k2}\right)} \|_{2}}r_{k}+\cdots+\omega\frac{e_{j_{kt}}}{\| A_{\left(j_{kt}\right)} \|_{2}}\frac{ A^{T}_{(j_{kt})}}{ \| A_{\left(j_{kt}\right)} \|_{2}}r_{k}\notag\\
&=~x_{k}+\omega\widetilde{I}_{\mathcal{J}_{k}}\widetilde{A}^{T}_{\mathcal{J}_{k}}r_{k},\label{derive}
\end{align}
where
\begin{eqnarray*}
\widetilde{I}_{\mathcal{J}_{k}}&=&\left[\frac{e_{j_{k1}}}{\| A_{\left(j_{k1}\right)} \|_{2}}, \frac{e_{j_{k2}}}{\| A_{\left(j_{k2}\right)} \|_{2}}, \ldots, \frac{e_{j_{kt}}}{\| A_{\left(j_{kt}\right)} \|_{2}}\right]\in R^{n\times t},\\
\widetilde{A}^{T}_{\mathcal{J}_{k}}&=&\left[\frac{ A_{(j_{k1})}}{ \| A_{\left(j_{k1}\right)} \|_{2}}, \frac{ A_{(j_{k2})}}{ \| A_{\left(j_{k2}\right)} \|_{2}}, \ldots, \frac{ A_{(j_{kt})}}{ \| A_{\left(j_{kt}\right)} \|_{2}}\right]^{T}\in R^{t\times m},
\end{eqnarray*}
and $\omega$  is used to adjust the convergence of the method and satisfies
\[
\left\{\begin{array}{ll}{0<\omega<\frac{2}{\sigma^2_{\max}(\widetilde{A}_{\mathcal{J}_{k}})},} & {\text { if } \Delta<0}, \\ {0<\omega<\frac{1-\sqrt{\Delta}}{\sigma^2_{\max}(\widetilde{A}_{\mathcal{J}_{k}})}~ \text {or} ~ \frac{1+\sqrt{\Delta}}{\sigma^2_{\max}(\widetilde{A}_{\mathcal{J}_{k}})}<\omega<\frac{2}{\sigma^2_{\max}(\widetilde{A}_{\mathcal{J}_{k}})},} & {\text { otherwise, }}\end{array}\right.
\]
where $\Delta=1-\frac{\sigma^2_{\max}(\widetilde{A}_{\mathcal{J}_{k}})\|A\|^2_F}{|\mathcal{J}_{k}|\sigma^{2}_{\min}(A)}$. It should be pointed out that the above conditions are only sufficient  but not necessary ones.

In summary,  we have \cref{alg3}.
\begin{algorithm}[!htb]
\caption{The PGBGS method}
\label{alg3}
\begin{algorithmic}
\STATE{\mbox{Input:} ~$A\in R^{m\times n}$, $b\in R^{m}$, $\theta \in[0,1]$, $\omega$, $\ell$, 
initial estimate $x_0\in R^{n}$}
\STATE{\mbox{Output:} ~$x_\ell$}
\FOR{$k=0, 1, 2, \ldots, \ell-1$}
\STATE{Compute
\begin{align}
\varepsilon_{k}=\frac{\theta}{\left\|A^Tr_k\right\|_{2}^{2}} \max \limits _{1 \leq j  \leq n}\left\{\frac{\left|A^T_{(j)}r_k\right|^{2}}{\left\|A_{\left(j\right)}\right\|_{2}^{2}}\right\}+\frac{1-\theta }{\|A\|_{F}^{2}} .\notag
\end{align}}
\STATE{Determine the index set of positive integers
\begin{align}
 \mathcal{J}_{k}=\left\{j_k \Bigg|  \left|A^T_{(j_k)}r_k\right|^{2} \geq \varepsilon_{k}\left\|A^Tr_k\right\|_{2}^{2}  \left\|A_{\left(j_k\right)}\right\|_{2}^{2}\right\}.\notag
\end{align}}
\STATE{Set
\begin{align}
\label{up3}
x_{k+1}=x_{k}+\omega\widetilde{I}_{\mathcal{J}_{k}}\widetilde{A}^{T}_{\mathcal{J}_{k}}r_{k}.
\end{align}}
\ENDFOR
\end{algorithmic}
\end{algorithm}

\begin{remark}
\label{rmk5}
To compare the computation cost of the GRCD, GBGS and PGBGS methods, we present the flops of the update rules of the three methods in \cref{TAB1}, where we have used the fact $A^{\dag}_{\mathcal{J}_{k}}=(A^T_{\mathcal{J}_{k}}A_{\mathcal{J}_{k}})^{-1}A^T_{\mathcal{J}_{k}}$.

\begin{table}[!htbp]\centering
\begin{small}\scriptsize
\caption{Flops for the update rules of the \texttt{GRCD}, \texttt{GBGS} and \texttt{PGBGS} methods.}\centering \label{TAB1}
 \begin{tabular}{|c|c|c|c|}
 \hline
\textbf{Method}  & \textbf{GRCD}     & \textbf{GBGS} &\textbf{PGBGS} \\
\hline
\textbf{Flops}  & $2(m+n)$     &$2|\mathcal{J}_{k}|^3+(4m-3)|\mathcal{J}_{k}|^2+(m+2n)|\mathcal{J}_{k}|$   &$(2m+2n+1)|\mathcal{J}_{k}|$ \\
\hline
\end{tabular}
\end{small}
\end{table}

It is easy to find that $2|\mathcal{J}_{k}|^3+(4m-3)|\mathcal{J}_{k}|^2+(m+2n)|\mathcal{J}_{k}|>(2m+2n+1)|\mathcal{J}_{k}|$. 
So the GBGS method needs more flops compared with the PGBGS method. 
Thus, for the same accuracy, the latter needs less computing time. More importantly, from the derivation of \cref{derive}, we can find that the PGBGS  method can be used for distributed implementations, which can yield more significant improvements in computation cost.
Although the GRCD method requires fewer flops in each iteration compared with our two block methods, the GRCD method needs much more iterations because it only executes one column in one iteration. So our methods behave better in the total computing time. Moreover, if $|\mathcal{J}_{k}|$ is small, the difference between the flops needed in the GRCD and PGBGS may be negligible.  These results are confirmed by experimental results given in \cref{sec5}.
\end{remark}

In the following, we give the convergence theorem of the PGBGS method.
\begin{theorem}
\label{theorem2}
The iteration sequence $\{x_k\}_{k=0}^\infty$ generated by \cref{alg3}, starting from an initial guess $x_0\in R^{n}$, converges linearly to the unique least squares solution $x_{\star}=A^{\dag}b$ and
\begin{align}
\left\|x_{k+1}-x_{\star}\right\|^2_{A^{T}A}&\leq~\left(1-(2\omega-\omega^2\sigma^2_{\max}(\widetilde{A}_{\mathcal{J}_{k}})) |\mathcal{J}_{k}| \frac{\sigma^{2}_{\min}(A)}{\left\|A\right\|_{F}^{2}}\right)\|x_k-x_\star\|^{2}_{A^{T}A}.\label{4sd}
\end{align}
\end{theorem}

\begin{proof}
From the update rule \cref{up3} in \cref{alg3}, we have
\begin{align}
A(x_{k+1}-x_{\star})&=~A(x_{k}-x_{\star})+\omega\widetilde{A}_{\mathcal{J}_{k}}\widetilde{A}^{T}_{\mathcal{J}_{k}}r_{k},  \notag
\end{align}
which together with the fact $A^TAx_{\star}=A^Tb$ gives
\begin{align}
A(x_{k+1}-x_{\star})&=~\left(I-\omega\widetilde{A}_{\mathcal{J}_{k}}\widetilde{A}^{T}_{\mathcal{J}_{k}} \right)A\left(x_{k}-x_\star\right). \notag
\end{align}
Taking the square of the Euclidean norm on both sides, we get
\begin{align}
\|A(x_{k+1}-x_{\star})\|^2_2&= ~\|\left(I-\omega\widetilde{A}_{\mathcal{J}_{k}}\widetilde{A}^{T}_{\mathcal{J}_{k}} \right)A\left(x_{k}-x_\star\right)\|^2_2\notag\\
&= ~\|A\left(x_{k}-x_\star\right)\|^2_2-2\omega\|\widetilde{A}^{T}_{\mathcal{J}_{k}}A\left(x_{k}-x_\star\right)\|^2_2+\omega^2\|\widetilde{A}_{\mathcal{J}_{k}}\widetilde{A}^{T}_{\mathcal{J}_{k}}A\left(x_{k}-x_\star\right)\|^2_2\notag\\
&\leq~\|A\left(x_{k}-x_\star\right)\|^2_2-(2\omega-\omega^2\sigma^2_{\max}(\widetilde{A}_{\mathcal{J}_{k}}))\|\widetilde{A}^{T}_{\mathcal{J}_{k}}A\left(x_{k}-x_\star\right)\|^2_2,\notag
\end{align}
or equivalently,
\begin{equation}
\label{754r}
\left\|x_{k+1}-x_{\star}\right\|^2_{A^{T}A}\leq\left\|x_{k}-x_{\star}\right\|^2_{A^{T}A}-(2\omega-\omega^2\sigma^2_{\max}(\widetilde{A}_{\mathcal{J}_{k}}))\|\widetilde{A}^{T}_{\mathcal{J}_{k}}A\left(x_{k}-x_\star\right)\|^2_2.
\end{equation}
On the other hand, from \cref{alg3} and \cref{derive}, we have
\begin{align}
\|\widetilde{A}^{T}_{\mathcal{J}_{k}}A\left(x_{k}-x_\star\right)\|^2_2
&=~\sum\limits_{j_k\in  \mathcal{J}_{k}} \left|\widetilde{A}^T_{(j_k)}r_k\right|^2=\sum\limits_{j_k\in  \mathcal{J}_{k}} \left|\frac{A^T_{(j_k)}}{\|A_{(j_k)\|_2}}r_k\right|^2\notag\\
&\geq~\sum\limits_{j_k\in  \mathcal{J}_{k}}\frac{1}{\left\|A_{\left(j_k\right)}\right\|_{2}^{2}} \varepsilon_{k} \left\|A^Tr_k\right\|_{2}^{2} \left\|A_{\left(j_k\right)}\right\|_{2}^{2}\notag\\
&=~\sum\limits_{j_k\in  \mathcal{J}_{k}} \left(  \theta  \max \limits _{1 \leq j  \leq n}\left\{\frac{\left|A^T_{(j)}r_k\right|^{2}}{\left\|A_{\left(j\right)}\right\|_{2}^{2}}\right\}+(1-\theta)\frac{ \left\|A^Tr_k\right\|_{2}^{2}}{\|A\|_{F}^{2}} \right) \notag\\
&\geq~|\mathcal{J}_{k}|\left(\theta\sum\limits_{j=1}^n \frac{\left|A^T_{(j)}r_k\right|^{2}}{\left\|A_{\left(j\right)}\right\|_{2}^{2}} \frac{\left\|A_{\left(j\right)}\right\|_{2}^{2}}{\|A\|^2_F}+(1-\theta)\frac{ \left\|A^Tr_k\right\|_{2}^{2}}{\|A\|_{F}^{2}}\right)\notag\\
&=~|\mathcal{J}_{k}|\frac{\|A^Tr_k\|^{2}_{2}}{\|A\|^2_F},\notag
\end{align}
which together with \cref{lem1} gives
\begin{align}
\|\widetilde{A}^{T}_{\mathcal{J}_{k}}A\left(x_{k}-x_\star\right)\|^2_2&\geq~|\mathcal{J}_{k}|   \frac{\sigma^{2}_{\min}(A)}{\left\|A\right\|_{F}^{2}} \|x_k-x_\star\|^{2}_{A^{T}A}.\label{12ww111}
\end{align}
Thus, substituting \cref{12ww111} into \cref{754r}, we obtain
the estimate \cref{4sd}.
\end{proof}

\section{Experimental results}\label{sec5}
In this section, we compare the GRCD, GBGS, and PGBGS methods  
using the matrix $A\in R^{m\times n}$ from two sets. One is generated randomly by using the MATLAB function \texttt{randn}, and the other one contains the matrices \texttt{abtaha1} and \texttt{ash958} from the University of Florida sparse matrix collection \cite{Davis2011}. To compare these methods fairly, 
we set $\theta=\frac{1}{2}$ in $\varepsilon_{k}$, i.e., let the index sets of the three methods be the same. In addition, we set $\omega=1$ in the PGBGS method.


We compare the three methods mainly in terms of the flops 
(denoted as ``Flops''), the iteration numbers (denoted as ``Iteration'') and the computing time in seconds (denoted as ``CPU time(s)''). In our specific experiments, the solution vector $x_\star$ is generated randomly by the MATLAB function \texttt{randn}. For the consistent problem, we set the right-hand side $b=Ax_{\star}$. For the inconsistent problem, we set the right-hand side $b=Ax_{\star}+r_{0}$, where $r_0$ is a nonzero vector belonging to the null space of $A^{T}$ generated by the MATLAB function \texttt{null}. All  the test problems are started from an initial zero vector $x_{0}=0$ and terminated once the \emph{relative solution error} (\texttt{RES}), defined by \[\texttt{RES}=\frac{\left\|x_{k}-x_{\star}\right\|^{2}_2}{\left\|x_{\star}\right\|^{2}_2},\] satisfies $\texttt{RES}<10^{-6}$ or the number of iteration steps exceeds $ 200,000$.

For the first class of matrices, that is, the matrices generated randomly, the numerical results on RES in base-10 logarithm versus the Flops, Iteration and CPU time(s) of the three methods are presented in \cref{fig1} when the linear system is consistent, and in \cref{fig2} when the linear system is inconsistent.

\begin{figure}[!htb]
 \begin{center}
\includegraphics [height=4cm,width=15cm ]{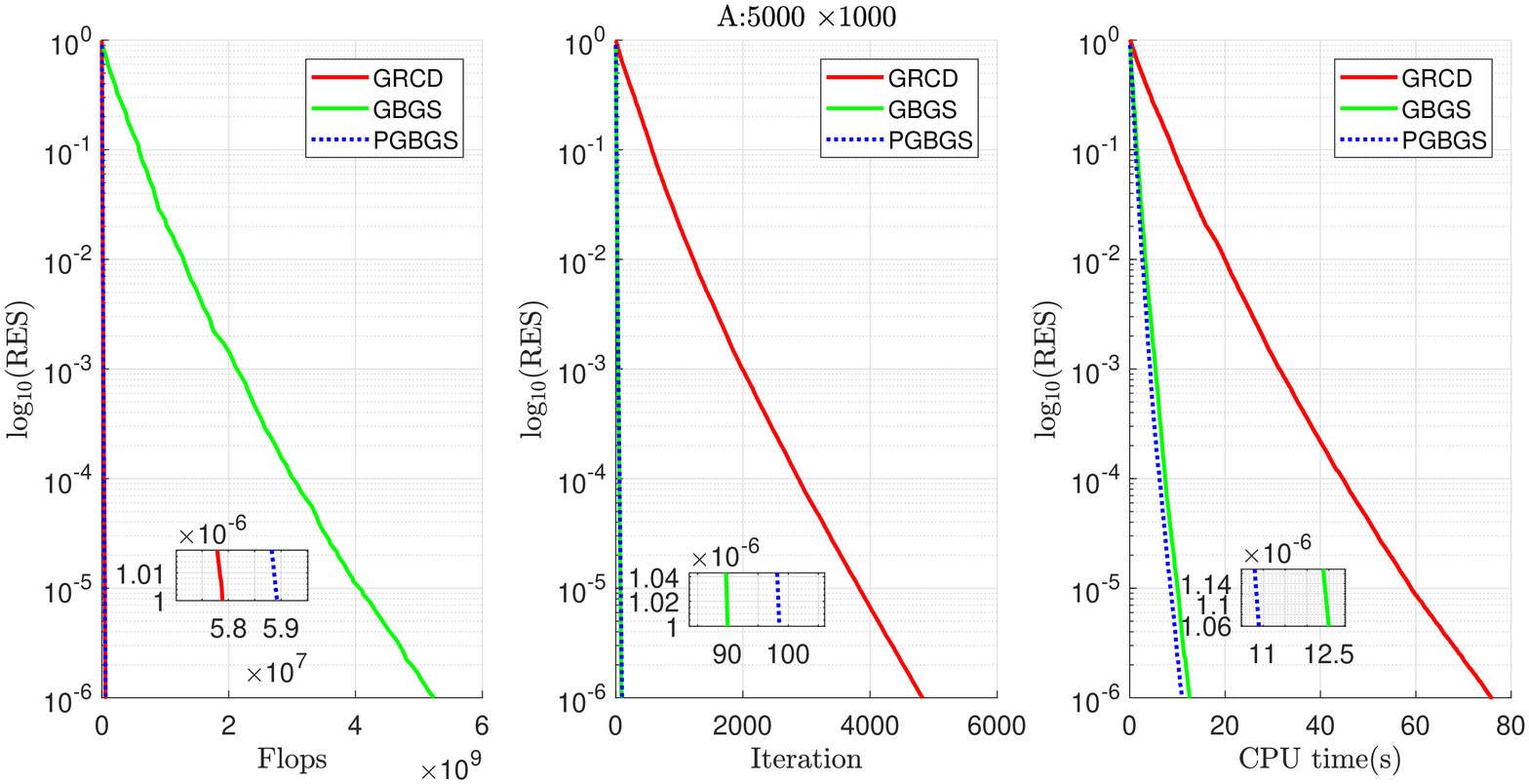}
\includegraphics [height=4cm,width=15cm ]{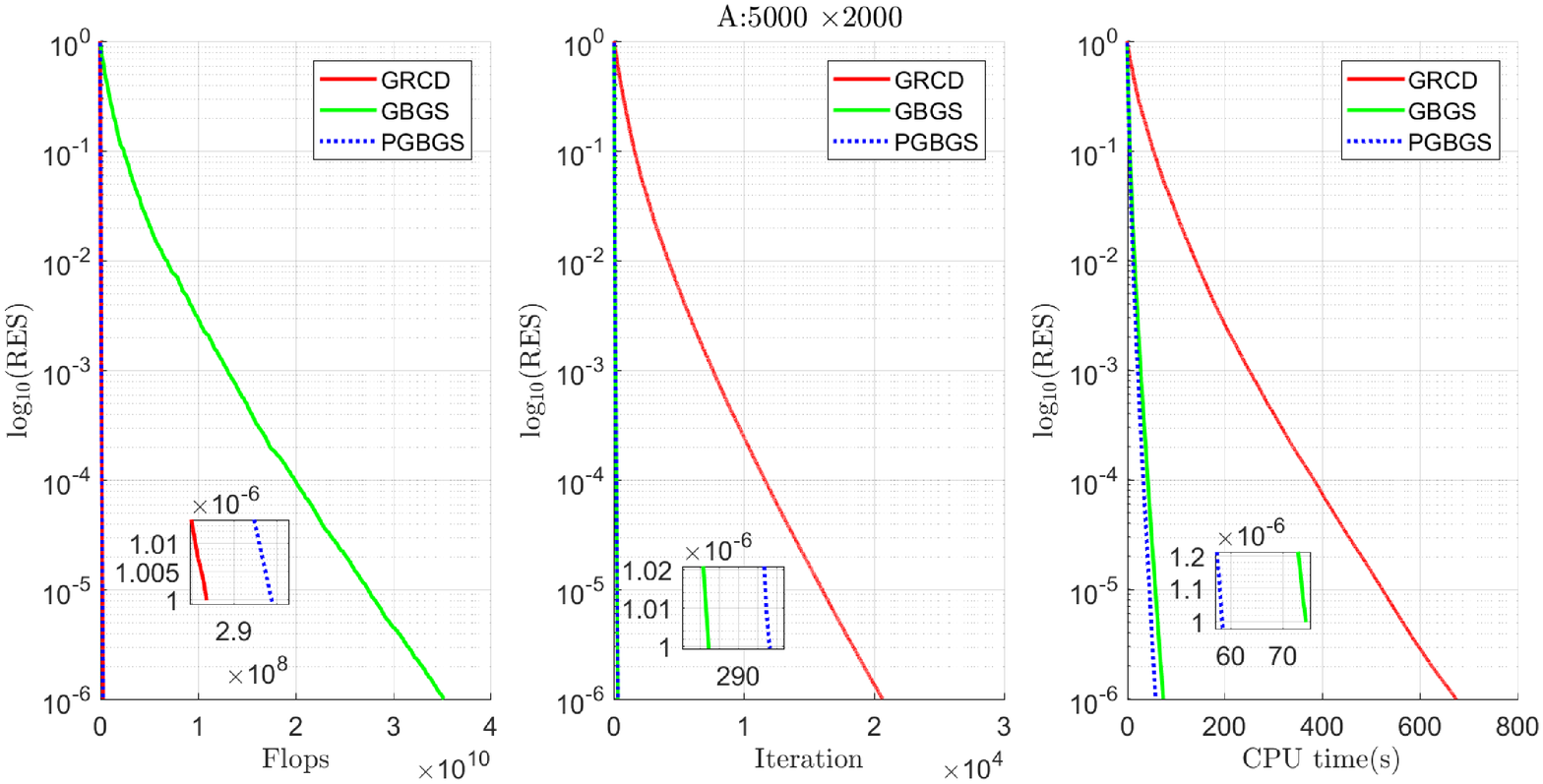}
 \end{center}
\caption{\rm Results for dense consistent linear systems with matrices generated randomly. (up) $A$ is of order $5000\times 1000$; (down) $A$ is of order $5000\times 2000$.}\label{fig1}
\end{figure}
\begin{figure}[!htb]
 \begin{center}
\includegraphics [height=4cm,width=15cm ]{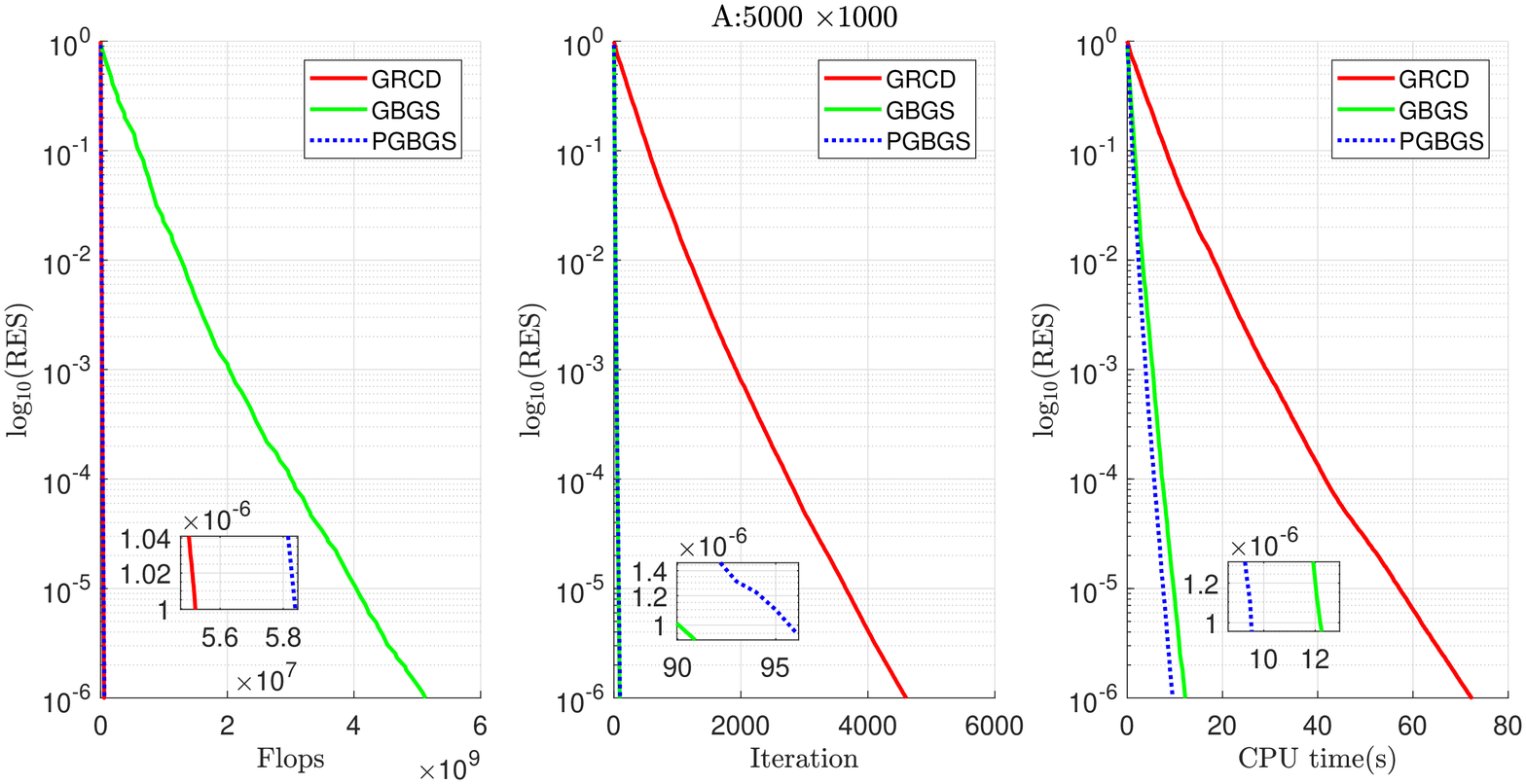}
\includegraphics [height=4cm,width=15cm ]{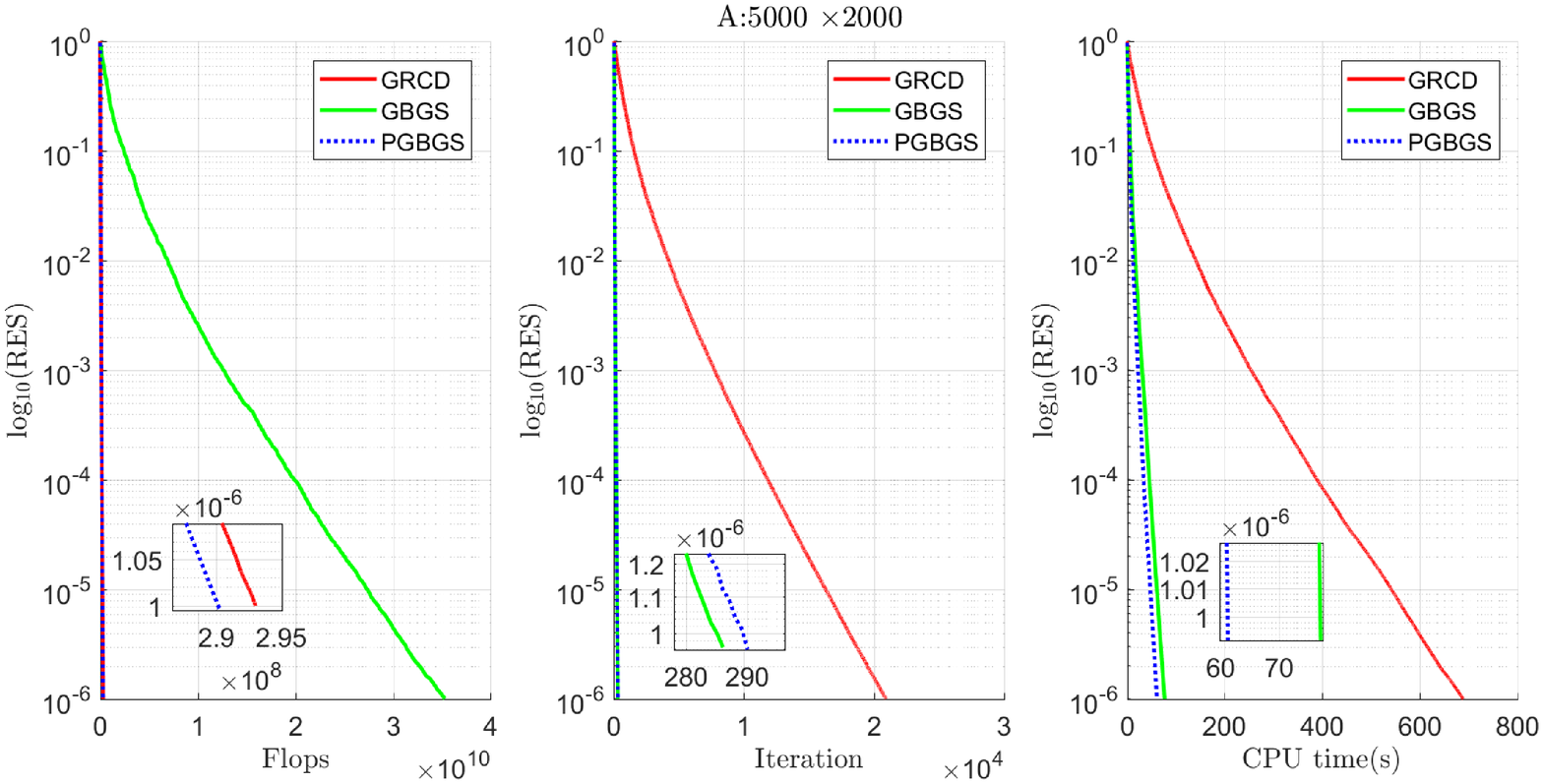}
 \end{center}
\caption{\rm Results for dense inconsistent linear systems with matrices generated randomly. (up) $A$ is of order $5000\times 1000$; (down) $A$ is of order $5000\times 2000$.}\label{fig2}
\end{figure}

From \cref{fig1,fig2}, we can find that the PGBGS method requires almost the same number of flops as the GRCD method, and the GBGS method needs the most flops. 
However, the GBGS method requires the fewest iterations, and, as desired, the PGBGS method needs the least computing time. Moreover, the differences in iterations and computing time between our methods and the GRCD method are remarkable. These results are consistent with the analysis in \cref{rmk5}. 

For the second class of matrices, that is, the sparse full column rank matrices from \cite{Davis2011}, we plot the numerical results on RES in base-10 logarithm versus the Flops, Iteration and CPU time(s) of the three methods in \cref{fig3} when the linear system is consistent, and in \cref{fig4} when the linear system is inconsistent.

\begin{figure}[!htb]
 \begin{center}
\includegraphics [height=4cm,width=15cm ]{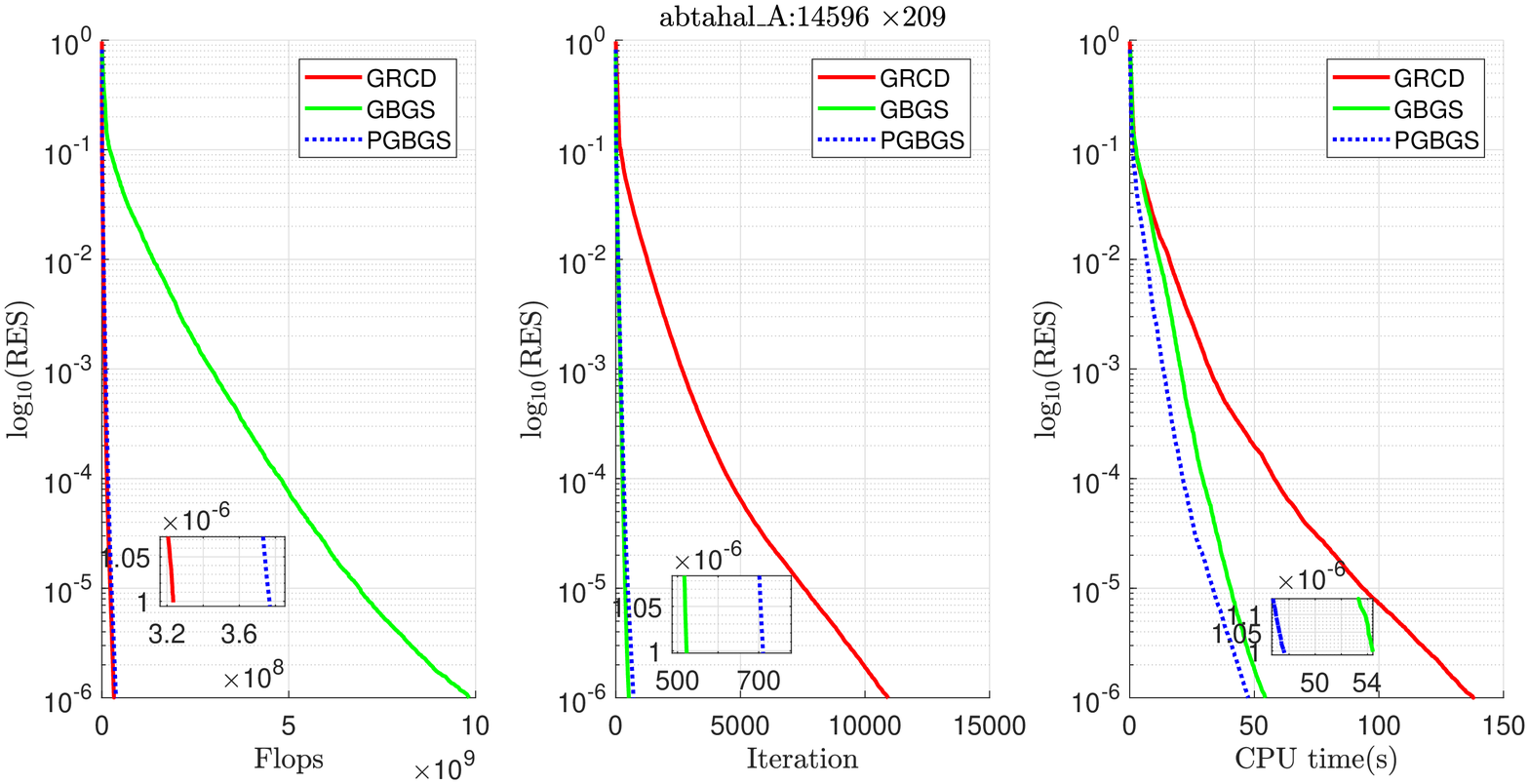}
\includegraphics [height=4cm,width=15cm ]{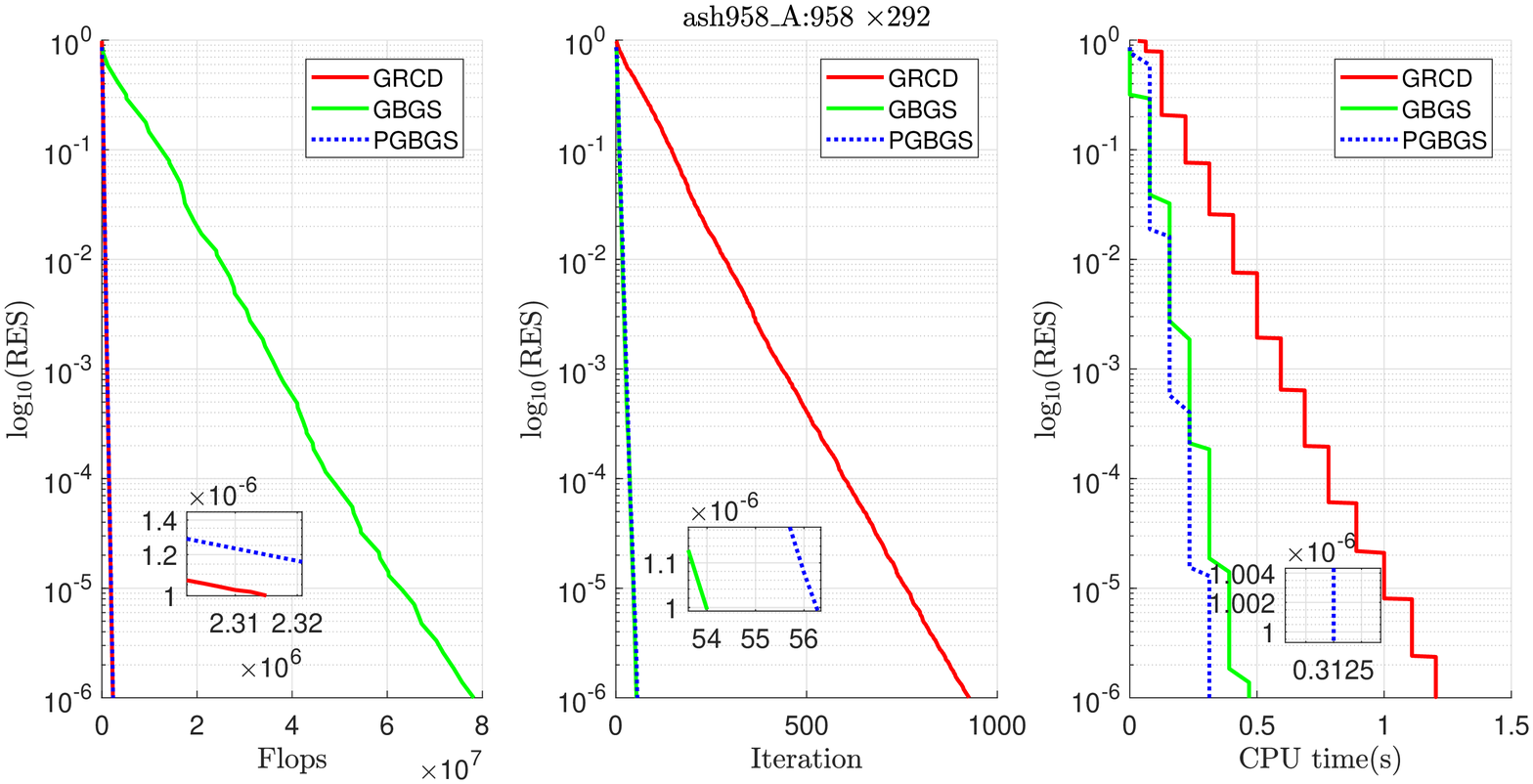}
 \end{center}
\caption{\rm Results for sparse consistent linear systems with matrices from the University of Florida sparse matrix collection. (up) \texttt{abtaha1}; (down) \texttt{ash958} .}\label{fig3}
\end{figure}
\begin{figure}[!htb]
 \begin{center}
\includegraphics [height=4cm,width=15cm ]{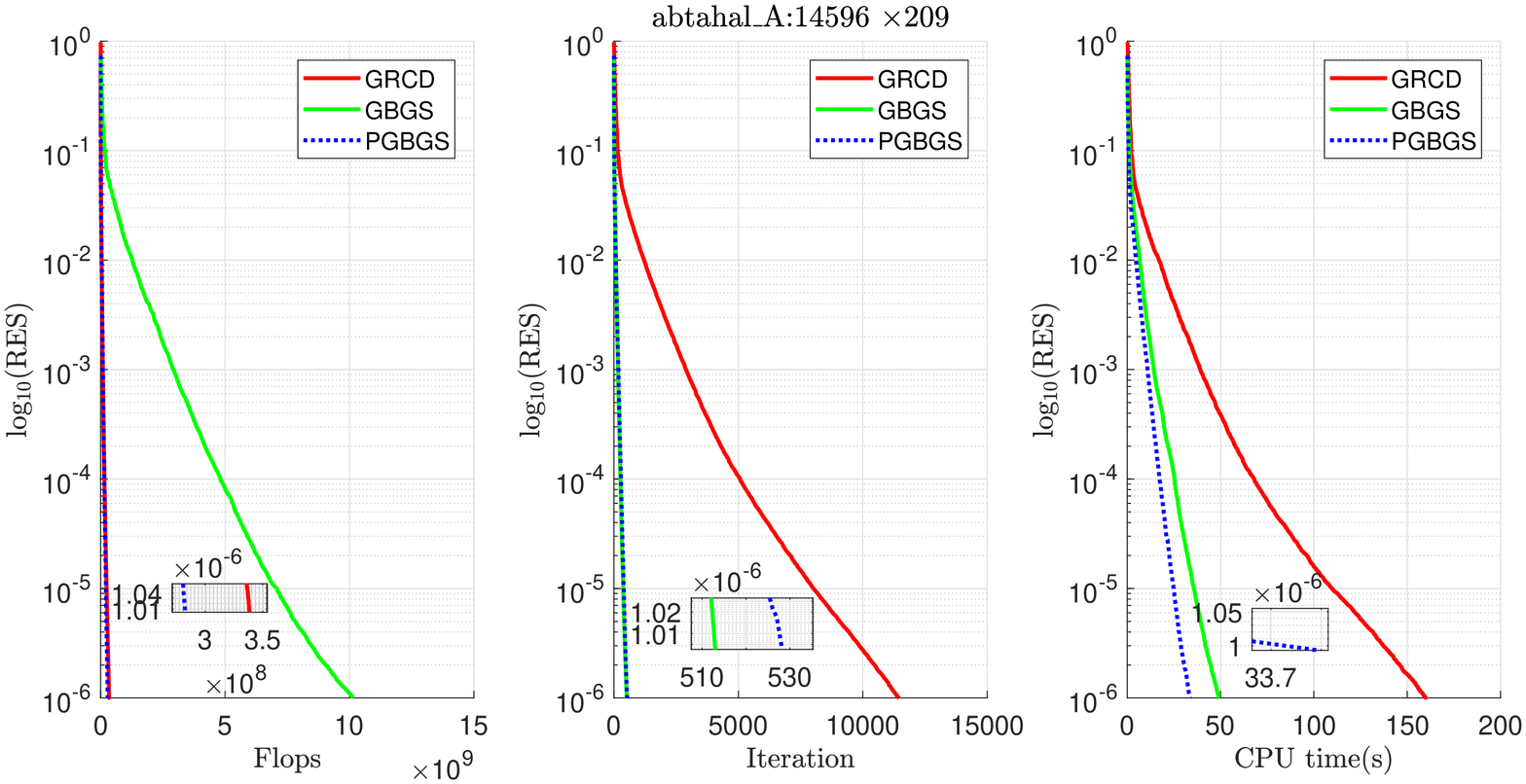}
\includegraphics [height=4cm,width=15cm ]{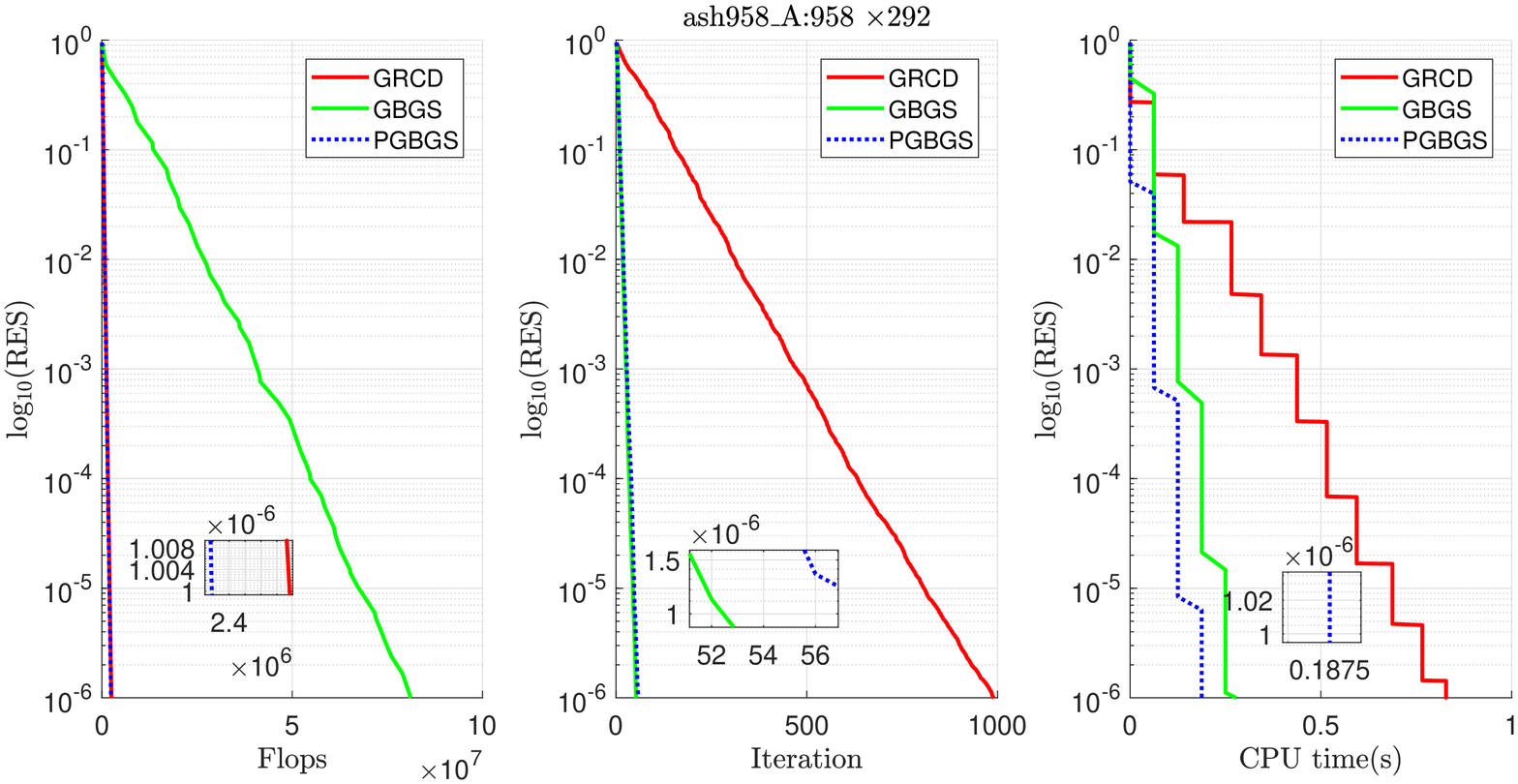}
 \end{center}
\caption{\rm Results for sparse inconsistent linear systems  with matrices from the University of Florida sparse matrix collection. (up) \texttt{abtaha1}; (down) \texttt{ash958}.}\label{fig4}
\end{figure}

The findings from \cref{fig3,fig4} are similar to the ones from \cref{fig1,fig2}. That is, the GBGS and PGBGS methods can far outperform the GRCD method in terms of the number of iterations and computing time.


Therefore, in all the cases, although the GBGS method needs more flops because of the computation of the Moore-Penrose pseudoinverse, it is still much faster than the GRCD method. This is mainly because the latter only updates one coordinate in each iteration while the former executes multiple coordinates. The PGBGS method has more advantages because it doesn't need to calculate the Moore-Penrose pseudoinverse any more.


 \clearpage
\bibliographystyle{siamplain}
\bibliography{references}

\begin{thebibliography}{10}

\bibitem{Bai2018}
{\sc Z.~Z. Bai and W.~T. Wu}, {\em {On greedy randomized Kaczmarz method for
  solving large sparse linear systems}}, SIAM J. Sci. Comput., 40 (2018),
  pp.~A592--A606.

\bibitem{Bai2018r}
{\sc Z.~Z. Bai and W.~T. Wu}, {\em {On relaxed greedy randomized Kaczmarz
  methods for solving large sparse linear systems}}, Appl. Math. Lett., 83
  (2018), pp.~21--26.

\bibitem{Bai2019}
{\sc Z.~Z. Bai and W.~T. Wu}, {\em {On greedy randomized coordinate descent
  methods for solving large linear least-squares problems}}, Numer. Linear
  Algebra Appl., 26 (2019), pp.~1--15.

\bibitem{Osborne1996}
{\sc A.~Bj$\ddot{o}$rck}, {\em {Numerical methods for least squares problems}},
  SIAM, Philadelphia., 1996.

\bibitem{Chen2017}
{\sc L.~Chen, D.~F. Sun, and K.~C. Toh}, {\em {An efficient inexact symmetric
  Gauss--Seidel based majorized ADMM for high-dimensional convex composite
  conic programming}}, Math. Program., 161 (2017), pp.~237--270.

\bibitem{Davis2011}
{\sc T.~A. Davis and Y.~F. Hu}, {\em {The university of florida sparse matrix
  collection}}, ACM. Trans. Math. Softw., 38 (2011), pp.~1--25.

\bibitem{Dukui2019}
{\sc K.~Du}, {\em {Tight upper bounds for the convergence of the randomized
  extended Kaczmarz and Gauss--Seidel algorithms}}, Numer. Linear Algebra
  Appl., 26 (2019), p.~e2233.

\bibitem{Du2019}
{\sc K.~Du and H.~Gao}, {\em {A new theoretical estimate for the convergence
  rate of the maximal weighted residual Kaczmarz algorithm}}, Numer. Math.
  Theor. Meth. Appl., 12 (2019), pp.~627--639.

\bibitem{Du20202}
{\sc K.~Du, W.~T. Si, and X.~H. Sun}, {\em Pseudoinverse-free randomized
  extended block kaczmarz for solving least squares}, arXiv preprint
  arXiv:2001.04179,  (2020).

\bibitem{du2019doubly}
{\sc K.~Du and X.~h. Sun}, {\em A doubly stochastic block gauss-seidel
  algorithm for solving linear equations}, arXiv preprint arXiv:1912.13291,
  (2019).

\bibitem{Edalatpour2017}
{\sc V.~Edalatpour, D.~Hezari, and D.~K. Salkuyeh}, {\em {A generalization of
  the Gauss--Seidel iteration method for solving absolute value equations}},
  Appl. Math. Comput., 293 (2017), pp.~156--167.

\bibitem{Griebel2012}
{\sc M.~Griebel and P.~Oswald}, {\em {Greedy and randomized versions of the
  multiplicative Schwarz method}}, Linear Algebra Appl., 437 (2012),
  pp.~1596--1610.

\bibitem{Haddock2019}
{\sc J.~Haddock and D.~Needell}, {\em {On Motzkin's method for inconsistent
  linear systems}}, BIT Numer. Math., 59 (2019), pp.~387--401.

\bibitem{Hefny2017}
{\sc A.~Hefny, D.~Needell, and A.~Ramdas}, {\em {Rows versus columns:
  Randomized Kaczmarz or Gauss--Seidel for ridge regression}}, SIAM J. Sci.
  Comput., 39 (2017), pp.~S528--S542.

\bibitem{Higham2002}
{\sc N.~J. Higham}, {\em {Accuracy and stability of numerical algorithms}},
  SIAM, Philadelphia., 2002.

\bibitem{Leventhal2010}
{\sc D.~Leventhal and A.~S. Lewis}, {\em {Randomized methods for linear
  constraints: Convergence rates and conditioning}}, Math. Oper. Res., 35
  (2010), pp.~641--654.

\bibitem{Liu2019}
{\sc Y.~Liu and C.~Q. Gu}, {\em {Variant of greedy randomized Kaczmarz for
  ridge regression}}, Appl. Numer. Math., 143 (2019), pp.~223--246.

\bibitem{Ma2015}
{\sc A.~Ma, D.~Needell, and A.~Ramdas}, {\em {Convergence properties of the
  randomized extended Gauss--Seidel and Kaczmarz methods}}, SIAM J. Matrix
  Anal. Appl., 36 (2015), pp.~1590--1604.

\bibitem{Necoara2019}
{\sc I.~Necoara}, {\em {Faster randomized block kaczmarz algorithms}}, SIAM J.
  Matrix Anal. Appl., 40 (2019), pp.~1425--1452.

\bibitem{Nguyen2017}
{\sc N.~Nguyen, D.~Needell, and T.~Woolf}, {\em {Linear convergence of
  stochastic iterative greedy algorithms with sparse constraints}}, IEEE Trans.
  Inf. Theory., 63 (2017), pp.~6869--6895.

\bibitem{Niu2020}
{\sc Y.~Q. Niu and B.~Zheng}, {\em {A greedy block Kaczmarz algorithm for
  solving large--scale linear systems}}, Appl. Math. Lett., 104 (2020),
  p.~106294.

\bibitem{Nutini2018}
{\sc J.~Nutini}, {\em {Greed is good: greedy optimization methods for
  large-scale structured problems}}, PhD thesis, University of British
  Columbia. 2018.

\bibitem{Osborne1961}
{\sc E.~E. Osborne}, {\em {On least squares solution of linear equations}}, J.
  Assoc. Comput. Mach., 8 (1961), pp.~628--636.

\bibitem{Razaviyayn2019}
{\sc M.~Razaviyayn, M.~Hong, N.~Reyhanian, and Z.~Q. Luo}, {\em {A linearly
  convergent doubly stochastic Gauss--Seidel algorithm for solving linear
  equations and a certain class of over--parameterized optimization problems}},
  Math. Program., 176 (2019), pp.~465--496.

\bibitem{Rebrova2019}
{\sc E.~Rebrova and D.~Needell}, {\em {Sketching for Motzkin's iterative method
  for linear systems}}, Proc. 50th Asilomar Conf. on Signals, Systems and
  Computers., 2019.

\bibitem{Saad2003}
{\sc Y.~Saad}, {\em {Iterative methods for sparse linear systems}}, SIAM,
  Philadelphia., 2003.

\bibitem{Tian2017}
{\sc Z.~L. Tian, M.~Y. Tian, Z.~Y. Liu, and T.~Y. Xu}, {\em {The Jacobi and
  Gauss--Seidel--type iteration methods for the matrix equation $AXB=C$}},
  Appl. Math. Comput., 292 (2017), pp.~63--75.

\bibitem{Tu2017}
{\sc S.~Tu, S.~Venkataraman, A.~C. Wilson, A.~Gittens, M.~I. Jordan, and
  B.~Recht}, {\em {Breaking locality accelerates block Gauss-Seidel}}, in
  ICML., 70 (2017), pp.~3482--3491.

\bibitem{WEIMAGGIE2018CONVERGENCEOT}
{\sc W.~M. Wu}, {\em Convergence of the randomized block gauss-seidel method},
  2018.

\bibitem{Xu1}
{\sc Y.~Y. Xu}, {\em {Hybrid Jacobian and Gauss--Seidel proximal block
  coordinate update methods for linearly constrained convex programming}}, SIAM
  J. Optimization., 28 (2018), pp.~646--670.

\bibitem{Zhang2019}
{\sc J.~J. Zhang}, {\em {A new greedy Kaczmarz algorithm for the solution of
  very large linear systems}}, Appl. Math. Lett., 91 (2019), pp.~207--212.

\end{thebibliography}
\end{document}